\documentclass{amsart}

\usepackage{amsmath}
\usepackage{amsthm}
\usepackage{amsfonts}
\usepackage{amssymb}

\newcommand\R{{\mathbb{R}}}
\newcommand\C{{\mathbb{C}}}

\renewcommand\P{{\mathbf{P}}}
\newcommand\E{{\mathbf{E}}}

\newcommand\N{{\mathbf{N}}}
\newcommand\Var{\mathbf{Var}}
\renewcommand\Im{{\operatorname{Im}}}
\renewcommand\Re{{\operatorname{Re}}}
\newcommand\eps{{\varepsilon}}

\newcommand\tr{\operatorname{trace}}

\newcommand\op{\operatorname{op}}

\newcommand\Dyson{{\operatorname{Sine}}}
\renewcommand\th{{\operatorname{th}}}
\renewcommand\sc{{\operatorname{sc}}}

\newcommand\condone{{{\bf C1}}}

\subjclass{15A52}

\theoremstyle{plain}
  \newtheorem{theorem}{Theorem}

  \newtheorem{proposition}[theorem]{Proposition}
  
  \newtheorem{lemma}[theorem]{Lemma}
  \newtheorem{corollary}[theorem]{Corollary}

\theoremstyle{definition}
  \newtheorem{definition}[theorem]{Definition}

\include{psfig}

\begin{document}

\title[Individual eigenvalue gap]{The asymptotic distribution of a single eigenvalue gap of a Wigner matrix}

\author{Terence Tao}
\address{Department of Mathematics, UCLA, Los Angeles CA 90095-1555}
\email{tao@math.ucla.edu}
\thanks{T. Tao is supported by NSF grant DMS-0649473.}

\begin{abstract}  We show that the distribution of (a suitable rescaling of) a single eigenvalue gap $\lambda_{i+1}(M_n)-\lambda_i(M_n)$ of a random Wigner matrix ensemble in the bulk is asymptotically given by the Gaudin-Mehta distribution, if the Wigner ensemble obeys a finite moment condition and matches moments with the GUE ensemble to fourth order.  This is new even in the GUE case, as prior results establishing the Gaudin-Mehta law required either an averaging in the eigenvalue index parameter $i$, or fixing the energy level $u$ instead of the eigenvalue index.

The extension from the GUE case to the Wigner case is a routine application of the Four Moment Theorem.  The main difficulty is to establish the approximate independence of the eigenvalue counting function $N_{(-\infty,x)}(\tilde M_n)$ (where $\tilde M_n$ is a suitably rescaled version of $M_n$) with the event that there is no spectrum in an interval $[x,x+s]$, in the case of a GUE matrix.  This will be done through some general considerations regarding determinantal processes given by a projection kernel.
\end{abstract}

\maketitle

\setcounter{tocdepth}{1}

\section{Introduction}

Given an $n \times n$ Hermitian matrix $M_n$, we let
$$ \lambda_1(M_n) \leq \ldots \leq \lambda_n(M_n)$$
be the $n$ eigenvalues of $M_n$ in non-decreasing order, counting multiplicity.  The purpose of this paper is to study the eigenvalue gaps $\lambda_{i+1}(M_n) - \lambda_i(M_n)$ of such matrices when $M_n$ is drawn from the Gaussian Unitary Ensemble (GUE), or more generally from a Wigner random matrix ensemble, in the asymptotic limit $n \to \infty$ and for a single $i = i(n)$ in the bulk region $\eps n \leq i \leq (1-\eps) n$.  

To begin with, let us set out our notational conventions for GUE and Wigner ensembles:

\begin{definition}[Wigner and GUE]\label{def:Wignermatrix}  Let $n \geq 1$ be an integer (which we view as a parameter going off to infinity). An $n \times n$ \emph{Wigner Hermitian matrix} $M_n$ is defined to be a  random Hermitian $n \times n$ matrix $M_n = (\xi_{ij})_{1 \leq i,j \leq n}$, in which the $\xi_{ij}$ for $1 \leq i \leq j \leq n$ are jointly independent with $\xi_{ji} = \overline{\xi_{ij}}$ (in particular, the $\xi_{ii}$ are real-valued), and each $\xi_{ij}$ has mean zero and variance one.  We say that the Wigner matrix ensemble \emph{obeys condition {\condone} with constant $C_0$} if one has
$$ \sup_{i,j} \E |\xi_{ij}|^{C_0} \leq C$$
for some constant $C$ (independent of $n$).

A \emph{GUE matrix} $M_n$ is a Wigner Hermitian matrix in which $\xi_{ij}$ is drawn from the complex gaussian distribution $N(0,1)_\C$ (thus the real and imaginary parts are independent copies of $N(0,1/2)_\R$) for $i \neq j$, and $\xi_{ii}$ is drawn from the real gaussian distribution $N(0,1)_\R$.

A Wigner matrix $M_n = (\xi_{ij})_{1 \leq i,j \leq n}$ is said to \emph{match moments to $m^{\th}$ order} with another Wigner matrix $M'_n= (\xi'_{ij})_{1\leq i,j \leq n}$ for some $m \geq 1$ if one has 
$$\E (\Re \xi_{ij})^a (\Im \xi_{ij})^b = \E (\Re \xi'_{ij})^a (\Im \xi'_{ij})^b$$
whenever $a,b \in \N$ with $a+b \leq m$.
\end{definition}

The bulk distribution of the eigenvalues $\lambda_1(M_n),\ldots,\lambda_n(M_n)$ of a Wigner (and in particular, GUE) matrix is governed by the \emph{Wigner semicircle law}.  Indeed, if we let $N_I(M_n)$ denote the number of eigenvalues of $M_n$ in an interval $I$, and we assume Condition {\condone} for some $C_0>2$, then with probability\footnote{See Section \ref{notation-sec} for the conventions for asymptotic notation such as $o(1)$ that are used in this paper.} $1-o(1)$, we have the asymptotic
$$ N_{\sqrt{n} I}(M_n) = n \int_I \rho_\sc(u)\ du + o(n)$$
uniformly in $I$, where $\rho_\sc$ is the Wigner semi-circular distribution
$$ \rho_\sc(u) := \frac{1}{2\pi} (4-u^2)_+^{1/2};$$
see e.g. \cite{AGZ}.
Informally, this law indicates that the eigenvalues $\lambda_i(M_n)$ are mostly contained in the interval $[-2\sqrt{n},2\sqrt{n}]$, and for any energy level $u$ in the bulk region $-2+\eps \leq u \leq 2-\eps$ for some fixed $\eps>0$, the average eigenvalue spacing should be $\frac{1}{\sqrt{n} \rho_\sc(u)}$ near $\sqrt{n} u$.

Now let $M_n$ be drawn from GUE.  The distribution of the eigenvalues $\lambda_1(M_n),\ldots,\lambda_n(M_n)$ are then well-understood.  If we define the $k$-point correlation functions $\rho^{(n)}_k: \R^k \to \R^+$ for $0 \leq k \leq n$ to be the unique symmetric continuous function for which
$$ \E \sum_{1 \leq i_1 < \ldots < i_k \leq n} F(\lambda_{i_1}(M_n),\ldots,\lambda_{i_k}(M_n)) = \int_{\R^k} F(x_1,\ldots,x_k) \rho^{(n)}_k(x_1,\ldots,x_k)\ dx_1 \ldots dx_k$$
for any continuous function $F$ which is compactly supported in the region $\{ x_1 \leq \ldots \leq x_k \}$, then one has the well-known formula of Dyson \cite{dyson}
$$ \rho^{(n)}_n(x_1,\ldots,x_n) = \frac{1}{(2\pi)^{n/2}} e^{-\sum_{i=1}^n x_i^2/2} \prod_{1 \leq i < j \leq n} (x_i-x_j)^2$$
and the \emph{Gaudin-Mehta formula}
$$ \rho^{(n)}_k(x_1,\ldots,x_k) = \det( K^{(n)}(x_i,x_j) )_{1 \leq i,j \leq k}$$
where $K^{(n)}(x,y)$ is the kernel
\begin{equation}\label{kform}
 K^{(n)}(x,y) := \sum_{k=0}^{n-1} P_k(x) e^{-x^2/4} P_k(y) e^{-y^2/4}
\end{equation}
and $P_0(x), P_1(x), \ldots$ are the $L^2$-normalised orthogonal polynomials with respect to the measure $e^{-x^2/2}\ dx$ (and are thus essentially Hermite polynomials); see e.g. \cite{Meh} or \cite{AGZ}.  In particular, the functions $P_k(x) e^{-x^2/4}$ for $i=0,\ldots,n-1$ are an orthonormal basis to the subspace $V^{(n)}$ of $L^2(\R) = L^2(\R, dx)$ spanned by $x^i e^{-x^2/4}$ for $i=0,\ldots,n-1$, thus the orthogonal projection $P^{(n)}$ to this subspace is given by the formula
$$P^{(n)} f(x) = \int_\R K^{(n)}(x,y) f(y)\ dy$$
for any $f \in L^2(\R)$.

Applying the inclusion-exclusion formula, the Gaudin-Mehta formula implies that for any interval\footnote{One can generalise this formula from intervals to arbitrary Borel measurable sets, but in our applications we will only need the interval case.  Similarly for many of the other determinantal process identities used in this paper.}  $I$, the probability $\P(N_I(M_n)=0)$ that $M_n$ has no eigenvalues in $I$, where $N_I(M_n)$ is the number of eigenvalues of $M_n$ in $I$, is equal to 
$$ \P(N_I(M_n)=0) = \sum_{k=0}^n \frac{(-1)^k}{k!} \int_I \ldots \int_I \det( K^{(n)}(x_i,x_j) )_{1 \leq i,j \leq k}\ dx_1 \ldots dx_k.$$
One can also express this probability as a Fredholm determinant
$$ \P(N_I(M_n)=0) = \det( 1 - 1_I P^{(n)} 1_I ),$$
where we view the indicator function $1_I$ as a multiplier operator on $L^2(\R)$.

The asymptotics of $K^{(n)}$ as $n \to \infty$ are also well understood, especially in the bulk of the spectrum\footnote{For the edge of the spectrum, one can control individual eigenvalues instead by the Tracy-Widom law \cite{TW-first}.  There is however a transitional regime between the bulk and the edge which is not covered by either our results or by the Tracy-Widom law, e.g. when $\min(i,n-i)$ is comparable to $n^\theta$ for some fixed $0<\theta<1$, and which may be worth further attention.  Given that convergence to the sine kernel is also known in such regimes, one would expect the main results of this paper to extend to these settings, although to make this intuition rigorous would require some careful argument which we do not pursue here.}, which in this normalisation corresponds to the interval $[(-2+\eps)\sqrt{n}, (2-\eps)\sqrt{n}]$ for any fixed $\eps>0$.  Indeed, if $-2+\eps < u < 2+\eps$ and $x, y$ are bounded uniformly in $n$, then from the Plancherel-Rotarch asymptotics for Hermite polynomials one has 
\begin{equation}\label{asym}
\frac{1}{\rho_\sc(u) \sqrt{n}} K^{(n)}( u \sqrt{n} + \frac{x}{\rho_\sc(u) \sqrt{n}}, u \sqrt{n} + \frac{y}{\rho_\sc(u) \sqrt{n}} ) = K_\Dyson(x,y) + o(1)
\end{equation}
where $K_\Dyson$ is the Dyson sine kernel
$$ K_\Dyson(x,y) := \frac{\sin(\pi(x-y))}{\pi(x-y)}$$
with the usual convention that $K(x,x)=1$; see e.g. \cite{Meh}, \cite{AGZ}, or \cite[Corollary 1]{dy}.  Note that the normalisations in \eqref{asym} are consistent with the heuristic, from the Wigner semi-circular law, that the mean eigenvalue spacing at $u\sqrt{n}$ is $\frac{1}{\rho_\sc(u)\sqrt{n}}$.  We observe that the Dyson sine kernel is also the kernel to the orthogonal projection $P_\Dyson$ to those functions $f \in L^2(\R, dx)$ whose Fourier transform
$$ \hat f(\xi) := \int_\R e^{-2\pi i x \xi} f(x)\ dx$$
is supported on the interval $[-1/2,1/2]$.

From \eqref{asym} and some careful treatment of error terms (see e.g. \cite[Chapter 3]{AGZ}) one obtains that
$$ \P( N_{u \sqrt{n} + \frac{1}{\rho_\sc(u) \sqrt{n}} I}(M_n)=0 ) =
\sum_{k=0}^\infty \frac{(-1)^k}{k!} \int_I \ldots \int_I \det( K_\Dyson(x_i,x_j) )_{1 \leq i,j \leq k}\ dx_1 \ldots dx_k + o(1),$$
or in Fredholm determinant form,
$$ \P( N_{u \sqrt{n} + \frac{1}{\rho_\sc(u) \sqrt{n}} I}(M_n)=0 ) = \det( 1 - 1_I P_\Dyson 1_I ) + o(1).$$
Note that the kernel $K_\Dyson(x,y) 1_I(y)$ of $P_\Dyson 1_I$ is square-integrable, and so $P_\Dyson 1_I$ is in the Hilbert-Schmidt class, and so $1_I P_\Dyson 1_I = (P_\Dyson 1_I)^* (P_\Dyson 1_I)$ is trace class.

This asymptotic can in turn be used to control the distribution of the averaged gap spacing distribution.  Indeed, if $1 < t_n < n$ is any sequence such that $1/t_n, t_n/n = o(1)$, then for any $-2+\eps < u < 2-\eps$ and $s > 0$ independent of $n$, the quantity
$$ S(s,t_n,u,M_n) := \frac{\# \{ 1 \leq i \leq n-1: \lambda_{i+1}(M_n)-\lambda_i(M_n) \leq \frac{s}{\sqrt{n} \rho_\sc(u)}; |\lambda_i(M_n)-u\sqrt{n}| \leq \frac{t_n}{\sqrt{n}}\}}{2t_n}$$
has the asymptotic
\begin{equation}\label{stum}
 S(s,t_n,u,M_n) = \int_0^s p(y)\ dy + o(1)
\end{equation}
where $p$ is the \emph{Gaudin distribution} (or \emph{Gaudin-Mehta distribution})
$$ p(y) := \frac{d^2}{dy^2} \det(1 - 1_{[0,y]} P_\Dyson 1_{[0,y]}),$$
or equivalently
\begin{equation}\label{equiv}
 \det(1 - 1_{[0,y]} P_\Dyson 1_{[0,y]}) = \int_y^\infty p(z) (z-y)\ dz;
 \end{equation}
see \cite{DKMVZ} for details.  The quantity $\det(1 - 1_{[0,y]} P_\Dyson 1_{[0,y]})$ (and hence the Gaudin distribution $p(y)$) can also be expressed in terms of a solution to a Painlev\'e V ordinary differential equation.  More precisely, one has
$$ \det(1 - 1_{[0,y]} P_\Dyson 1_{[0,y]}) = \exp\left( \int_0^{\pi y} \frac{\sigma(x)}{x}\ dx\right)$$
where $\sigma$ solves the ODE
$$ (x\sigma'')^2 + 4 (x\sigma'-\sigma)(x\sigma'-\sigma+(\sigma')^2) = 0$$
with boundary condition $\sigma(x) \sim -\frac{x}{\pi}$ as $x \to 0$; see \cite{JMM} (or the later treatment in \cite{TW}).  Among other things, this implies that the Gaudin distribution $p$ and all of its derivatives are\footnote{Indeed, the famous \emph{Wigner surmise} predicts the reasonably accurate approximation $p(x) \approx \frac{1}{2} \pi x e^{-\pi x^2/4}$; see e.g. \cite{Meh}.} smooth, bounded, and rapidly decreasing on $(0,+\infty)$.

We also remark that the extreme values of the gaps $\lambda_{i+1}(M_n)-\lambda_i(M_n)$ are also well understood; see \cite{bb}.  However, our focus here will be on the bulk distribution of these gaps rather than on the tail behaviour.

In \cite{TVlocal1}, a \emph{Four Moment Theorem} for the eigenvalues of Wigner matrices was established, which roughly speaking asserts that the fine scale statistics of these eigenvalues depend only on the first four moments of the coefficients of the Wigner matrix, so long as some decay condition (such as Condition {\condone}) is obeyed.  In particular, by applying this theorem to the asymptotic \eqref{stum} for GUE matrices, one obtains

\begin{corollary}  The asymptotic \eqref{stum} is also valid for Wigner matrices $M_n$ which obey Condition {\condone} for some sufficiently large absolute constant $C_0$, and which match moments with GUE to fourth order.
\end{corollary}

\begin{proof} See \cite[Theorem 9]{TVlocal1}.  Strictly speaking, the arguments in that paper require an exponential decay hypothesis on the coefficients on $M_n$ rather than a finite moment condition, because the four moment theorem in that paper also has a similar requirement.  However, the refinement to the four moment theorem established in the subsequent paper \cite{TVlocal3} (or in the later papers \cite{TVeigenvector}, \cite{knowles}) relaxes that exponential decay condition to a finite moment condition.
\end{proof}

We remark that the moment matching hypothesis in this corollary can in fact be removed by combining the above argument with some similar results obtained (by a different method) in \cite{Joh1}, \cite{EPRSY}; see \cite{ERSTVY}.

The Wigner semi-circle law predicts that the location of an individual eigenvalue $\lambda_i(M_n)$ of a Wigner or GUE matrix $M_n$ for $i$ in the bulk region $\eps n \leq i \leq (1-\eps) n$ should be approximately $\sqrt{n} u$, where $u = u_{i/n}$ is the \emph{classical location} of the eigenvalue, given by the formula
\begin{equation}\label{classical}
 \int_{-\infty}^u \rho_\sc(y)\ dy = \frac{i}{n}.
\end{equation}
Indeed, it is a result of Gustavsson \cite{Gus} that $\frac{\lambda_i(M_n) - \sqrt{n} u}{\sqrt{\log n/2\pi^2}/\sqrt{n} \rho_\sc(u)}$ converges in distribution to the standard real Gaussian distribution $N(0,1)_\R$, or more informally that
\begin{equation}\label{heuris}
 \lambda_i(M_n) \approx N( \sqrt{n} u, \frac{\log n/2\pi^2}{(\sqrt{n} \rho_\sc(u))^2} ).
\end{equation}
Note that the standard deviation $\frac{\sqrt{\log n/2\pi^2}}{\sqrt{n} \rho_\sc(u)}$ here exceeds the mean eigenvalue spacing $\frac{1}{\sqrt{n} \rho_\sc(u)}$ by a factor comparable to $\sqrt{\log n}$.  If one heuristically applies this approximation \eqref{heuris} to the gap distribution law 
\eqref{stum}, one is led to the conjecture that the normalised eigenvalue gap
$$ \frac{\lambda_{i+1}(M_n) - \lambda_i(M_n)}{1/(\sqrt{n} \rho_\sc(u))}$$
should converge in distribution to the Gaudin distribution, in the sense that
\begin{equation}\label{single-gap}
\P( \frac{\lambda_{i+1}(M_n) - \lambda_i(M_n)}{1/(\sqrt{n} \rho_\sc(u))} \leq s ) = \int_0^s p(y)\ dy + o(1)
\end{equation}
for any fixed $s>0$.

Unfortunately, this is not quite a rigorous proof of \eqref{single-gap}.  The problem is that the asymptotic \eqref{stum} involves not just a single eigenvalue gap $\lambda_{i+1}-\lambda_i$, but is instead an average over all eigenvalue gaps near the energy level $\sqrt{n} u$.  By \eqref{heuris}, one is then forced to consider the contributions of at least $\gg \sqrt{\log n}$ different values of $i$ that could contribute to \eqref{stum}.  One would of course expect the behaviour of $\lambda_{i+1}-\lambda_i$ for adjacent values of $i$ to be essentially identical, in which case one could pass from the averaged gap distribution \eqref{stum} to the individual gap distribution \eqref{single-gap}.  However, it is \emph{a priori} conceivable (though admittedly quite strange) that there is non-trivial dependence on $i$, for instance that $\lambda_{i+1}-\lambda_i$ might tend to be larger than predicted by the Gaudin distribution for even $i$, and smaller than predicted for odd $i$, with the two effects canceling out in averaged statistics such as \eqref{stum}, but not in non-averaged statistics such as \eqref{single-gap}.

Our main result rules out such a pathological possibility:

\begin{theorem}[Individual gap spacing]\label{gue-main}  Let $M_n$ be drawn from GUE, and let $\eps n \leq i \leq (1-\eps) n$ for some fixed $\eps>0$.  Then one has the asymptotic \eqref{single-gap} for any fixed $s>0$, where $u = u_{i/n}$ is given by \eqref{classical}.
\end{theorem}

Applying the four moment theorem from \cite{TVlocal1} (with the extension to the finite moment setting in \cite{TVlocal3}), one obtains an immediate corollary:

\begin{corollary}\label{wigner-main}  The conclusion of Theorem \ref{gue-main} is also valid for Wigner matrices $M_n$ which obey Condition {\condone} for some sufficiently large absolute constant $C_0$, and which match moments with GUE to fourth order.
\end{corollary}

\begin{proof}  This can be established by repeating the proof of \cite[Theorem 9]{TVlocal1} (in fact the argument is even simpler than this, because one is working with a single eigenvalue gap rather than with an average, and can proceed more analogously to the proof of \cite[Corollary 21]{TVlocal1}).  We omit the details.
\end{proof}

In view of the results in \cite{ERSTVY}, it is natural to conjecture that the moment matching condition can be removed.  Following \cite{ERSTVY}, it would be natural to use heat flow methods to do so, in particular by trying to extend Theorem \ref{gue-main} to the \emph{gauss divisible} ensembles studied in \cite{Joh1}.  However, the methods in this paper rely very heavily on the determinantal form of the joint eigenvalue distribution of GUE (and not just on control of the $k$-point correlation functions); the formulae in \cite{Joh1} also have some determinantal structure, but it is unclear to us whether this similarity of structure is sufficient to replicate the arguments\footnote{By using heat flow methods such as the method of local relaxation flow \cite{ESY}, one can obtain control on energy-averaged correlation functions in this setting, and similarly for non-classical $\beta$-ensembles $\beta \neq 1,2,4$ as was done recently in \cite{bourgade}.  Such bounds are sufficient to obtain averaged gap information of the form \eqref{stum} (at least for values of $t_n$ that grow faster than logarithmic), but it is not obvious how to isolate a single eigenvalue gap to then obtain \eqref{single-gap}.}.  On the other hand, we expect analogues Theorem \ref{gue-main} to be establishable for other ensembles with a determinantal form, such as GOE and GSE, or to more general $\beta$ ensembles involving a non-quadratic potential for the classical values $1,2,4$ of $\beta$.  We will not pursue these matters here.

The key to proving Theorem \ref{gue-main} lies in establishing the approximate independence\footnote{It may be surprising to the experts that the counting functions on $(-\infty,x)$ and $[x,x+s]$ are approximately independent, as the intervals are adjacent.  The point is that while there is a correlation between the two counting functions, the covariance between them is essentially of order $O(1)$, whilst the variance of $N_{(-\infty,x)}$ is of order $\log n$, and so the correlation between the two random variables is ends up being asymptotically negligible. To put it another way, most of the random fluctuation of $N_{(-\infty,x)}$ comes from the portion of the spectrum that is far away from $x$, and this contribution will be almost completely decoupled from the spectrum at $[x,x+s]$.} of the eigenvalue counting function $N_{(-\infty,x)}(\tilde M_n)$ from the event that $\tilde M_n$ has no eigenvalues in a short interval $[x,x+s]$ (i.e. that $N_{[x,x+s]}(\tilde M_n)=0$), where $\tilde M_n$ is a suitably rescaled version of $M_n$.  Roughly speaking, this independence, coupled with a central limit theorem for $N_{(-\infty,x)}(\tilde M_n)$, will imply that the distribution of a gap $\lambda_{i+1}(M_n)-\lambda_i(M_n)$ is essentially invariant with respect to small changes in the $i$ parameter.  To obtain this approximate independence, we use the properties of determinantal processes, and in particular the fact that a determinantal point process $\Sigma$, when conditioned on the event that a given interval such as $[x,x+s]$ contains no elements of $\Sigma$, remains a determinantal point process (though with a slightly different kernel).  The main difficulty is then to ensure that the new kernel is close to the old kernel in a suitable sense (more specifically, we will compare the two kernels in the nuclear norm $S^1$).

We thank Peter Forrester, Van Vu, and the anonymous referee for corrections.

\section{Notation}\label{notation-sec}

In this paper, $n$ will be an asymptotic parameter going to infinity.  A quantity is said to be \emph{fixed} if it does not depend on $n$; if a quantity is not specified as fixed, then it is permitted to vary with $n$.  Given two quantities $X, Y$, we write $X = O(Y)$, $X \ll Y$, or $Y \gg X$ if we have $|X| \leq CY$ for some fixed $C$, and $X = o(Y)$ if $X/Y$ goes to zero as $n \to \infty$.

An \emph{interval} will be a connected subset of the real line, which may possibly be half-infinite or infinite.  If $I$ is an interval, we use $I^c := \R \backslash I$ to denote its complement.

We use $\sqrt{-1}$ to denote the imaginary unit, in order to free up the symbol $i$ for other purposes, such as indexing eigenvalues.

Given a bounded operator $A$ on a Hilbert space $H$, we denote the operator norm of $A$ as $\|A\|_{\op}$. We will also need the \emph{Hilbert-Schmidt norm} (or \emph{Frobenius norm})
$$ \|A\|_{HS} := (\tr(A^* A))^{1/2}= (\tr(A A^*))^{1/2},$$
with the convention that this norm is infinite if $A^* A$ or $AA^*$ is not trace class.  Similarly, we will need the \emph{Schatten $1$-norm} (or \emph{nuclear norm})
$$ \|A\|_{S^1} := \tr( (A^* A)^{1/2} ) = \tr( (AA^*)^{1/2} ),$$
which is finite when $A$ is trace class.  Note that if $A$ is compact with non-zero singular values $\sigma_1,\sigma_2,\ldots$ then we have
\begin{align*}
\|A\|_{\op} &= \sup_i |\sigma_i| \\
\|A\|_{HS} &= (\sum_i |\sigma_i|^2)^{1/2} \\
\|A\|_{S^1} &= \sum_i |\sigma_i|.
\end{align*}
Indeed, one should view the operator, Hilbert-Schmidt, and nuclear norms as non-commutative versions of the $\ell^\infty$, $\ell^2$, and $\ell^1$ norms respectively.

For us, the reason for introducing the nuclear norm $S^1$ is that it controls the trace:
$$ |\tr A| \leq \|A\|_{S^1}.$$
On the other hand, the Hilbert-Schmidt and operator norms are significantly easier to estimate than the nuclear norm. To bridge the gap, we will rely heavily on the non-commutative H\"older inequalities
\begin{align*}
 \|AB\|_{\op} &\leq \|A\|_{\op} \|B\|_{\op} \\
 \|AB\|_{HS} &\leq \|A\|_{\op} \|B\|_{HS} \\
 \|AB\|_{HS} &\leq \|A\|_{HS} \|B\|_{\op} \\
 \|AB\|_{S^1} &\leq \|A\|_{\op} \|B\|_{S^1} \\
 \|AB\|_{S^1} &\leq \|A\|_{S^1} \|B\|_{\op} \\
 \|AB\|_{S^1} &\leq \|A\|_{HS} \|B\|_{HS};
\end{align*}
see e.g. \cite{bhatia}.  We will use these inequalities in this paper without further comment.

We remark that for integral operators
$$ Tf(x) := \int_\R K(x,y) f(y)\ dy$$
on $L^2(\R)$ for locally integrable $K$, the Hilbert-Schmidt norm of $T$ is given by
$$ \|T\|_{HS} = (\int_\R \int_\R |K(x,y)|^2\ dx dy)^{1/2}$$
when the right-hand side is finite.

\section{Some general theory of determinantal processes}

In this section we record some of the theory of determinantal processes which we will need.  We will not attempt to exhaustively describe this theory here, referring the interested reader to the surveys \cite{lyons}, \cite{sos} or \cite{hkpv} instead.  We will also not aim for maximum generality in this section, restricting attention to determinantal processes on $\R$, whose associated locally trace class operator $P$ will usually be an orthogonal projection, and often of finite rank.

Define a \emph{good kernel} to be a locally integrable function  $K: \R \times \R \to \C$, such that the associated integral operator
$$ P f(x) := \int_\R K(x,y) f(y)\ dy$$
can be extended from $C_c(\R)$ to a self-adjoint bounded operator on $L^2(\R)$, with spectrum in $[0,1]$.  Furthermore, we require that $P$ be locally trace class in the sense that for every compact interval $I$, the operator $1_I P 1_I$ is trace class; this will for instance be the case if $K$ is smooth.  If $K$ is a good kernel, then (as was shown in \cite{macchi}, \cite{sos}; see also \cite{hkpv} or \cite{AGZ}), $K$ defines a point process $\Sigma\subset \R$, i.e. a random subset\footnote{Strictly speaking, a point process is permitted to have multiplicity, so that it becomes a multiset rather than a set.  However, as we are restricting attention to kernels $K$ which are locally integrable, the determinantal point processes we consider will be almost surely simple, in the sense that no multiplicity occurs.} of $\R$ that is almost surely locally finite, with the $k$-point correlation functions
\begin{equation}\label{xk}
 \rho_k(x_1,\ldots,x_k) := \det(K(x_i,x_j))_{1 \leq i,j \leq k}
\end{equation}
for any $k \geq 0$, thus
$$ \E \prod_{i=1}^k \#(\Sigma \cap I_i) = \int_{I_1 \times \ldots \times I_k} \rho_k(x_1,\ldots,x_k)\ dx_1 \ldots dx_k$$
for any disjoint intervals $I_1,\ldots,I_k$.  This process is known as the \emph{determinantal point process} with kernel $K$.  

The distribution of a determinantal point process in an interval $I$ is described by the following lemma:

\begin{lemma}\label{dpp}  Let $\Sigma$ be a determinantal point process on $\R$ associated to a good kernel $K$ and associated operator $P$.  Let $I$ be a compact interval, and suppose that the operator $1_I P 1_I$ has non-zero eigenvalues $\lambda_1, \lambda_2, \ldots \in (0,1]$.  Then $\# (\Sigma \cap I)$ has the same distribution as $\sum_i \xi_i$, where the $\xi_i$ are jointly independent Bernoulli random variables, with each $\xi_i$ equalling $1$ with probability $\lambda_i$ and $0$ with probability $1-\lambda_i$.  In particular, one has
$$ \E \# (\Sigma \cap I) = \sum_i \lambda_i = \tr(1_I P 1_I) $$
and
$$ \Var \# (\Sigma \cap I) = \sum_i (1-\lambda_i) \lambda_i = \tr( (1 - 1_I P 1_I) 1_I P 1_I ),$$
and
$$ \P( \# (\Sigma \cap I) = 0 ) = \prod_i (1-\lambda_i) = \det( 1 - 1_I P 1_I ).$$
\end{lemma}

\begin{proof} See e.g. \cite[Corollary 4.2.24]{AGZ}.
\end{proof}

As a corollary of Lemma \ref{dpp}, we see that $\P( \#(\Sigma \cap I) = 0 ) > 0$ unless $P$ has an eigenfunction of eigenvalue $1$ that is supported on $I$.

An important special case of determinantal point processes arises when the operator $P$ is an orthogonal projection of some finite rank $n$, which is the situation with the GUE point process $\{ \lambda_1(M_n),\ldots,\lambda_n(M_n)\}$, which as discussed in the introduction is a determinantal point process with kernel $K^{(n)}$ given by \eqref{kform}.  In this case, the hypotheses on $P$ (i.e. self-adjoint trace class with eigenvalues in $[0,1]$) are automatically satisfied, and the determinantal point process $\Sigma$ is almost surely a set of cardinality $n$; see e.g. \cite{sos}, \cite{hkpv} or \cite{AGZ}.  In this situation, the $k$-point correlation functions $\rho_k$ vanish for $k>n$, and for $k<n$ we have the \emph{Gaudin lemma}
\begin{equation}\label{gaudin-lemma}
\rho_k(x_1,\ldots,x_k) = \frac{1}{n-k} \int_\R \rho_{k+1}(x_1,\ldots,x_{k+1})\ dx_{k+1}
\end{equation}
which allows one to recursively obtain the correlation functions from the $n$-point correlation function $\rho_n$ (which is essentially the joint density function of the $n$ elements of $\Sigma$).  Note that \eqref{gaudin-lemma} in fact holds for any point process whose cardinality is almost surely $n$, if the process is almost surely simple with locally integrable correlation functions.

If $V$ is the $n$-dimensional range of $P$, and $\phi_1,\ldots,\phi_n$ is an orthonormal basis for $V$, then the kernel $K$ of the orthogonal projection $P$ can be expressed explicitly as
$$ K(x,y) = \sum_{i=1}^n \phi_i(x) \overline{\phi_i(y)}$$
and thus (by the basic formula $\det(A^* A)= |\det(A)|^2$)
\begin{equation}\label{rhodi}
 \rho_n(x_1,\ldots,x_n) = |\det( \phi_i(x_j) )_{1 \leq i,j \leq n}|^2.
\end{equation}

This leads to the following consequence:

\begin{proposition}[Exclusion of an interval]\label{exclude}  Let $\Sigma$ be a determinantal process associated to the orthogonal projection $P_V$ to an $n$-dimensional subspace $V$ of $L^2(\R)$.  Let $I$ be a compact interval, and suppose that no non-trivial element of $V$ is supported in $I$.  Then the event $E := ( \#(\Sigma \cap I) = 0 )$ occurs with non-zero probability, and upon conditioning to this event $E$, the resulting random variable $(\Sigma|E)$ is a determinantal point process associated to the orthogonal projection $P_{1_{I^c} V}$ to the $n$-dimensional subspace $1_{I^c} V$ of $L^2(\R)$.
\end{proposition}

\begin{proof}  This is a continuous variant of \cite[Proposition 6.3]{lyons}, and can be proven as follows.  By construction, $P_V$ has no eigenvector of eigenvalue $1$ supported in $I$, and so $\P(E) = \det(1  - 1_I P_V )$ is non-zero.  The point process $(\Sigma|E)$ clearly has cardinality $n$ almost surely, and is thus described by its $n$-point correlation function, which is a constant multiple of
$$ \rho_n(x_1,\ldots,x_n) 1_{I^c}(x_1) \ldots 1_{I^c}(x_n),$$
which by \eqref{rhodi} can be written as
\begin{equation}\label{dr}
 |\det( \phi_i 1_{I^c}(x_j) )_{1 \leq i,j \leq n}|^2,
\end{equation}
where $\phi_1,\ldots,\phi_n$ is an orthonormal basis for $V$.

By hypothesis on $V$, $\phi_1 1_{I^c},\ldots,\phi_n 1_{I^c}$ is a (not necessarily orthonormal) basis for $1_{I^c} V$.  By row operations, we can thus write \eqref{dr} as a constant multiple of
$$  |\det( \phi'_i(x_j) )_{1 \leq i,j \leq n}|^2,$$
where $\phi'_1,\ldots,\phi'_n$ is an orthonormal basis for $1_{I^c} V$.  But this is the $n$-point correlation function for the determinantal point process of $P_{1_{I^c} V}$.  As the $n$-point correlation function of an $n$-point process integrates to $n!$ (cf. \eqref{gaudin-lemma}), we see that the $n$-point correlation function of $(\Sigma|E)$ must be exactly equal to that of the determinantal point process of $P_{1_{I^c} V}$, as claimed.
\end{proof}

It is likely that the above proposition can be extended to infinite-dimensional projections (possibly after imposing some additional regularity hypotheses), but we will not pursue this matter here.

We saw in Lemma \ref{dpp} that if $\Sigma$ is a determinantal point process and $I$ is a compact interval, then the random variable $\# (\Sigma\cap I)$ is the sum of independent Bernoulli random variables.  If the variance of this sum is large, then such a sum should converge to a gaussian, by the central limit theorem.  Examples of such central limit theorems for $\# (\Sigma \cap I)$ were formalised in \cite{CLe}, \cite{sos-gauss}.  We will need a slight variant of these theorems, which gives uniform convergence on the probability density function of $\# (\Sigma \cap I)$ as opposed to the probability distribution function.

\begin{lemma}[Discrete density version of central limit theorem]\label{dbe}  Let $X = \xi_1 + \ldots + \xi_n$ be the sum of independent Bernoulli random variables $\xi_1,\ldots,\xi_n$, with mean $\E X = \mu$ and variance $\Var X = \sigma^2$ for some $\sigma > 0$.  Then for any integer $m$, one has
$$ \P( X = m ) = \frac{1}{\sqrt{2\pi} \sigma} e^{-(m-\mu)^2/2\sigma^2} + O( \sigma^{-1.7} ).$$
\end{lemma}

One can improve the error term here to $O(\sigma^{-2})$ with a little more effort, but any error better than $1/\sigma$ will suffice for our purposes.

\begin{proof}  We may assume that $\sigma$ is larger than any given absolute constant, as the claim is trivial otherwise.  We use the Fourier-analytic method.  Write $p_i := \E \xi_i$, then
$$ \mu = \sum_{i=1}^n p_i$$
and
\begin{equation}\label{sigma-form}
 \sigma^2 = \sum_{i=1}^n p_i (1-p_i).
\end{equation}
Observe that $X$ has characteristic function
$$ \E e^{2\pi \sqrt{-1} tX} =\prod_{i=1}^n ((1-p_i) + p_i e^{2\pi \sqrt{-1} t})$$
and so
$$ \P( X = m ) = \int_{-1/2}^{1/2} \prod_{i=1}^n ((1-p_i) e^{-2\pi \sqrt{-1} p_i t} + p_i e^{2\pi (1-p_i) \sqrt{-1} t}) e^{-2\pi \sqrt{-1} (m-\mu) t}\ dt.$$
We can rewrite this integral slightly as $A + B$, where
$$ A := \int_{|t| \leq \sigma^{-0.9}} \prod_{i=1}^n ((1-p_i) e^{-2\pi \sqrt{-1} p_i t} + p_i e^{2\pi \sqrt{-1} (1-p_i) t}) e^{-2\pi \sqrt{-1} (m-\mu) t}\ dt$$
and
$$ B := \int_{\sigma^{-0.9} < t \leq 1/2} \prod_{i=1}^n ((1-p_i) e^{-2\pi \sqrt{-1} p_i t} + p_i e^{2\pi \sqrt{-1} (1-p_i) t}) e^{-2\pi \sqrt{-1} (m-\mu) t}\ dt.$$

We first control $A$.  From Taylor expansion one has
$$ (1-p_i) e^{-2\pi \sqrt{-1} p_i t} + p_i e^{2\pi \sqrt{-1} (1-p_i) t} = \exp( - 2 \pi^2 p_i (1-p_i) t^2 + O( p_i (1-p_i) |t|^3 ) ) $$
in this regime, and so by \eqref{sigma-form}
$$ \prod_{i=1}^n (1-p_i) e^{-2\pi \sqrt{-1} p_i t} + p_i e^{2\pi \sqrt{-1} (1-p_i) t} = \exp( - 2 \pi^2 \sigma^2 t^2 + O( \sigma^{-0.7} ) ).$$
We therefore have
$$ A = \int_{-\sigma^{-0.9}}^{\sigma^{-0.9}} (1 + O(\sigma^{-0.7})) e^{-2\pi^2 \sigma^2 t^2} e^{-2\pi \sqrt{-1} (m-\mu) t}\ dt.$$
Since
$$ \int_\R e^{-2\pi^2 \sigma^2 t^2}\ dt = \frac{1}{\sqrt{2\pi} \sigma} e^{-(m-\mu)^2/2\sigma^2} $$
and
$$ \int_\R e^{-2\pi^2 \sigma^2 t^2} e^{-2\pi \sqrt{-1} (m-\mu) t}\ dt = 
\frac{1}{\sqrt{2\pi} \sigma} e^{-(m-\mu)^2/2\sigma^2} $$
and
$$ \int_{|t| \geq \sigma^{-0.9}} e^{-2\pi^2 \sigma^2 t^2}\ dt = O( \sigma^{-100} )$$
(say), we conclude that
\begin{equation}\label{A-form}
A =  \frac{1}{\sqrt{2\pi} \sigma} e^{-(m-\mu)^2/2\sigma^2} + O( \sigma^{-1.7} ).
\end{equation}

Now we control $B$.  Elementary computation shows that
$$ |(1-p_i) e^{-2\pi \sqrt{-1} p_i t} + p_i e^{2\pi \sqrt{-1} (1-p_i) t}| \leq \exp( - c p_i(1-p_i) t^2 )$$
in this regime for some absolute constant $c>0$.By \eqref{sigma-form}, we may thus bound
$$ |B| \leq \int_{|t| \geq \sigma^{-0.9}} e^{-c \sigma^2 t^2}\ dt$$
and so $B = O(\sigma^{-1.7})$.  Combining this with \eqref{A-form}, the claim follows.
\end{proof}

Combining Lemma \ref{dpp}, Proposition \ref{exclude}, and Lemma \ref{dbe} we immediately obtain

\begin{corollary}\label{core}  Let $\Sigma$ be a determinantal point process on $\R$ whose kernel $P$ is an orthogonal projection to an $n$-dimensional subspave $V$ of $L^2(\R)$. Let $I$ be a compact interval, and $m$ be an integer.  Then
$$ \P( \# (\Sigma \cap I) = m ) = \frac{1}{\sqrt{2\pi} \sigma} e^{-(m-\mu)^2/2\sigma^2} + O( \sigma^{-1.7} ).$$
where
$$ \mu := \tr(1_I P 1_I) $$
and
$$ \sigma^2 := \tr( (1 - 1_I P 1_I) (1_I P 1_I) ) .$$
Furthermore, if $J$ is another compact interval disjoint from $J$, such that no non-trivial element of $V$ is supported on $J$, then
$$ \P( \# (\Sigma \cap I) = m | \# (\Sigma \cap J) = 0 ) = \frac{1}{\sqrt{2\pi} \tilde \sigma} e^{-(m-\tilde \mu)^2/2\tilde \sigma^2} + O( \tilde \sigma^{-1.7} ),$$
where
$$ \tilde \mu := \tr(\tilde P 1_I)$$
and
$$ \tilde \sigma^2 := \tr( \tilde P 1_{I^c} \tilde P 1_I ),$$
and $\tilde P$ is the orthogonal projection to $1_{J^c} V$.
\end{corollary}

Let the notation be as in the above corollary.  In our application, we will need to determine the extent to which events such as $\# (\Sigma \cap I) = m$ and $\# (\Sigma \cap J) = 0$ are independent.  In view of Corollary \ref{core}, it is then natural to determine the extent to which the projection $\tilde P$ differs from that of $P$.

Observe that as no non-trivial element of $V$ is supported on $J$, the operator $P 1_{J^c} P$, viewed as a map from $V$ to $V$, is invertible.  Denoting its inverse on $V$ by $(P 1_{J^c} P)_V^{-1}$, we then see that the operator $1_{J^c} P (P 1_{J^c} P)_V^{-1} P 1_{J^c}$ is self-adjoint, idempotent, and has $1_{J^c} V$ as its range, and so must be equal to $\tilde P$:
\begin{equation}\label{pip-pop}
\tilde P := 1_{J^c} P (P 1_{J^c} P)_V^{-1} P 1_{J^c}.
\end{equation}
We can also write $(P 1_{J^c} P)_V^{-1}$ as $(1 - P 1_J P)_V^{-1}$ (since $P$ is the identity on $V$).  Thus, by Neumann series, we formally have the expansion
$$ \tilde P= 1_{J^c} P 1_{J^c} + 1_{J^c} P 1_J P 1_{J^c} + 1_{J^c} P 1_J P 1_J P 1_{J^c} + \ldots$$
This expansion is convergent for sufficiently small $J$, but does not necessarily converge for $J$ large.  However, in practice we will be able to invert $1-P1_J P$ by a perturbation argument involving the Fredholm alternative.  More precisely, in our application, the finite rank projection $P$ will be ``close'' in some weak sense to an infinite rank projection $P_0$ (in our application, $P_0$ will be the Dyson projection $P_\Dyson$), projecting to some infinite-dimensional Hilbert space $V_0$.  We will assume $P_0$ to be locally trace class, so that $P_0 1_J P_0$ is compact.  If we assume that no non-trivial element of $V_0$ is supported on $J$, then the Fredholm alternative (see e.g. \cite[Theorem VI.14]{rs}) then implies that $1 - P_0 1_J P_0: V_0 \to V_0$ is invertible, with inverse $(1 - P_0 1_J P_0)_{V_0}^{-1} = 1 + K_0$ for some compact operator $K_0$, thus
\begin{equation}\label{pok}
 (1 + K_0) (1-P_0 1_J P_0)  = 1
\end{equation}
on $L^2(\R)$.  As $1-P_0 1_J P_0$ is self-adjoint and is the identity on $V_0$, we see that $K_0$ has range in $V_0$ and cokernel in $V_0^\perp$, thus $$K_0= P_0 K_0 = K_0 P_0 = P_0 K_0 P_0.$$
One then expects $1 + PK_0P: V \to V$ to be an approximate inverse to $1-P1_J P$.  Indeed, we have
\begin{equation}\label{polio}
 (1+PK_0P) (1-P1_J P) = 1 + E
\end{equation}
where
$$ E := P K_0 P - P (1+K_0) P 1_J P.$$
Meanwhile, from \eqref{pok} we have
\begin{equation}\label{kooper}
 K_0 = (1+K_0) P_0 1_J P_0
\end{equation}
and thus
$$ E = P (1+K_0) (P_0 1_J P_0 - P 1_J P) P.$$
Let us now bound some norms of $E$.  As the projection operator $P$ has an operator norm of at most $1$, one has
$$ \|E\|_{\op} \leq (1 + \|K_0\|_{\op}) \| P_0 1_J P_0 - P 1_J P \|_{\op};$$
splitting $P_0 1_J P_0 - P 1_J P$ as $(P_0-P) 1_J P_0 - P 1_J (P_0 - P)$ we conclude that
$$ \|E\|_{\op} \leq (1 + \|K_0\|_{\op}) ( \| (P_0-P) 1_J P_0 \|_{\op} + \| (P_0-P) 1_J P \|_{\op} ).$$
If we now make the hypothesis that
\begin{equation}\label{hypot}
\| (P_0-P) 1_J \|_{\op} \leq \frac{1}{4(1+\|K_0\|_{\op})}
\end{equation}
then we have $\|E\|_{\op} \leq 1/2$, and so we have the Neumann series
$$ (1+E)^{-1} = 1 - E + E^2 - \ldots.$$
In particular,
$$ \| (1+E)^{-1} - 1 \|_{S^1} \leq 2 \|E\|_{S^1}.$$
To bound the right-hand side, we use the triangle inequality to obtain
$$ \|E\|_{S^1} \leq (1 + \|K_0\|_{\op}) (\|P_0 1_J P_0 \|_{S^1} + \|P 1_J P \|_{S^1}).$$
Factorising $P_0 1_J P_0  = (1_J P_0)^* (1_J P_0)$ and similarly for $P 1_J P$, we conclude that
$$ \| (1+E)^{-1} - 1 \|_{S^1} \leq 2 (1 + \|K_0\|_{\op}) (\|1_J P_0 \|_{HS}^2 + \|1_J P \|_{HS}^2).$$

Note that $E$ maps $V$ to itself, and so $(1+E)^{-1}$ can also be viewed as an operator from $V$ to itself (being the identity on $V^\perp$).  From \eqref{polio} one then has
$$ (P 1_{J^c} P)_V^{-1} = (1+E)^{-1} (1+PK_0P)$$
and thus
$$\| (P 1_{J^c} P)_V^{-1} - (1+PK_0P) \|_{S^1} \leq
2 (1 + \|K_0\|_{\op})^2 (\|1_J P_0 \|_{HS}^2 + \|1_J P \|_{HS}^2).$$
Applying \eqref{pip-pop}, we conclude that
$$ \| \tilde P - 1_{J^c} P (1+PK_0P) P 1_{J^c} \|_{S^1} \leq
2 (1 + \|K_0\|_{\op})^2 (\|1_J P_0 \|_{HS}^2 + \|1_J P \|_{HS}^2).$$
To deal with the $PK_0P$ term we observe from \eqref{kooper} and the factorisation $P_0 1_J P_0  = (1_J P_0)^* (1_J P_0)$
that
$$ \|K_0\|_{S^1} \leq (1+\|K_0\|_{\op}) \| 1_J P_0\|_{HS}^2,$$
and so
$$ \| \tilde P - 1_{J^c} P 1_{J^c} \|_{S^1} \leq
3 (1 + \|K_0\|_{\op})^2 (\|1_J P_0 \|_{HS}^2 + \|1_J P \|_{HS}^2).$$

We summarise the above discussion as a proposition:

\begin{proposition}[Approximate description of $\tilde P$]\label{app}  Let $P$ be a projection to an $n$-dimensional subspace $V$ of $L^2(\R)$, and let $J$ be a compact interval such that no non-trivial element of $V$ is supported on $J$.  Let $P_0$ be a projection to a (possibly infinite-dimensional) subspace $V_0$ of $L^2(\R)$ which is locally trace class, and such that no non-trivial element of $V_0$ is supported on $J$.  Let $K_0: L^2(\R) \to L^2(\R)$ be the compact operator solving \eqref{pok} that is provided by the Fredholm alternative.  Suppose that
\begin{equation}\label{hyp}
\| (P_0-P) 1_J \|_{\op} \leq \frac{1}{4 (1+\|K_0\|_{\op})}.
\end{equation}
Let $\tilde P$ be the orthogonal projection to $1_{J^c} V$.  Then
$$ \| \tilde P - 1_{J^c} P 1_{J^c} \|_{S^1} \leq 3 (1 + \|K_0\|_{\op})^2 (\|1_J P_0 \|_{HS}^2 + \|1_J P \|_{HS}^2).$$
\end{proposition}

Because the $S^1$ norm controls the trace, this proposition allows us to compare the quantities $\tilde \mu,\tilde \sigma^2$ from Corollary \ref{core} with their counterparts $\mu,\sigma^2$:

\begin{corollary}\label{closo}  Let $n, P, V, J, P_0, K_0$ be as in Proposition \ref{app} (in particular, we make the hypothesis \eqref{hyp}).  Let $I$ be a compact interval disjoint from $J$, and let $\mu,\sigma^2,\tilde \mu,\tilde \sigma^2$ be as in Corollary \ref{core}.  Then we have
$$ \tilde \mu = \mu + O( M ) $$
and
$$ \tilde \sigma^2 = \sigma^2 + O(M)$$
where $M$ is the quantity
$$M := (1 + \|K_0\|_{\op})^2 (\|1_J P_0 \|_{HS}^2 + \|1_J P \|_{HS}^2).$$
\end{corollary}

In practice, this corollary will allow us to show that the random variable $\#(\Sigma \cap I)$ is essentially independent of the event $\#(\Sigma \cap J)=0$ for certain determinantal point processes $\Sigma$ and disjoint intervals $I,J$.

\section{Proof of main theorem}

We are now ready to prove Theorem \ref{gue-main}.  We may of course assume that $n$ is larger than any given absolute constant.

Let $n, M_n, \eps, i, u$ be as in Theorem \ref{gue-main}, and let $X$ be the random variable
$$ X := \frac{\lambda_{i+1}(M_n) - \lambda_i(M_n)}{1/(\sqrt{n} \rho_\sc(u))}.$$
Clearly $X$ takes values in $\R^+$ almost surely.  Our task is to show that
$$ \P( X \leq s ) = \int_0^s p(y)\ dy + o(1)$$
for all fixed $s>0$, or equivalently that
\begin{equation}\label{paxis}
 \P( X > s ) = \int_s^\infty p(y)\ dy + o(1)
 \end{equation}
for all fixed $s>0$.

It will suffice to show that
\begin{equation}\label{xis}
 \E (X-s)_+ = \int_s^\infty (y-s) p(y)\ dy + o(1)
\end{equation}
for all fixed $s>0$, since on applying this with two choices $0 < s_1 < s_2$ of $s$, subtracting, and then dividing by $s_2-s_1$ we see that
$$ \E \min( \frac{(X-s_1)_+}{s_2-s_1}, 1) = \int_{s_1}^\infty \min(\frac{y-s_1}{s_2-s_1}, 1) p(y)\ dy + o(1);$$
letting $s_1, s_2$ approach a given value $s$ from the left or right, we then conclude the bounds
$$ \int_{s+\delta}^\infty p(y)\ dy - o(1) \leq \P( X > s ) \leq \int_{s-\delta}^\infty p(y)\ dy + o(1) $$
for any fixed $\delta>0$, and \eqref{paxis} follows from the monotone convergence theorem.

It remains to prove \eqref{xis}.  By \eqref{equiv}, the left-hand side of \eqref{xis} can be written as
$$ \det(1 - 1_{[0,s]} P_\Dyson 1_{[0,s]}) + o(1).$$
Meanwhile, if we introduce the normalised random matrix
$$ \tilde M_n := \frac{M_n - u \sqrt{n}}{1/\sqrt{n} \rho_\sc(u)}$$
then we have
$$ X = \lambda_{i+1}(\tilde M_n) - \lambda_i(\tilde M_n).$$
For any fixed choice of $\tilde M_n$, we observe the identity
$$ (X-s)_+ = \int_\R 1_{N_{(-\infty,x)}(\tilde M_n)=i \wedge N_{[x,x+s]}(\tilde M_n)=0}\ dx,$$
since the set of real numbers $x$ for which $N_{(-\infty,x)}(\tilde M_n)=i \wedge N_{[x,x+s]}(\tilde M_n)=0$ holds is an interval of length $X-s$ when $X>s$, and empty otherwise.  Taking expectations and using the Fubini-Tonelli theorem, we conclude that
$$ \E (X-s)_+ = \int_\R \P( N_{(-\infty,x)}(\tilde M_n)=i \wedge N_{[x,x+s]}(\tilde M_n)=0 )\ dx.$$
Our task is thus to show that
\begin{equation}\label{silly}
\int_\R \P( N_{(-\infty,x)}(\tilde M_n)=i \wedge N_{[x,x+s]}(\tilde M_n)=0 )\ dx 
= \det(1 - 1_{[0,s]} P_\Dyson 1_{[0,s]}) + o(1).
\end{equation}
Let $t_n := \log^{0.6} n$ (say).  We will shortly  establish the following claims:

\begin{enumerate}
\item[(i)] (Tail estimate)  We have
\begin{equation}\label{masx}
\int_{|x| \geq t_n}  \P( N_{(-\infty,x)}(\tilde M_n)=i )\ dx = o(1).
\end{equation}
\item[(ii)] (Approximate independence)  For $|x| < t_n$, one has
\begin{equation}\label{indigo}
\P( N_{(-\infty,x)}(\tilde M_n)=i \wedge N_{[x,x+s]}(\tilde M_n)=0 ) = 
\P( N_{(-\infty,x)}(\tilde M_n)=i ) \P( N_{[x,x+s]}(\tilde M_n)=0 ) + O( \log^{-0.85} n ).
\end{equation}
\item[(iii)]  (Gap probability at fixed energy)  For $|x| < t_n$, one has
\begin{equation}\label{violet}
 \P( N_{[x,x+s](\tilde M_n)=0} ) = \det(1 - 1_{[0,s]} P_\Dyson 1_{[0,s]}) + o(1).
\end{equation}
\item[(iv)]  (Central limit theorem)  For $|x| < t_n$, one has
\begin{equation}\label{paprika}
 \P( N_{(-\infty,x)}(\tilde M_n)=i ) = \frac{1}{\sqrt{2\pi} \sigma} e^{-x^2/2\sigma^2} 
+ O( \log^{-0.85} n )
\end{equation}
where $\sigma := \sqrt{\log n/2\pi^2}$.
\end{enumerate}

Let us assume these estimates for the moment.  From \eqref{indigo}, \eqref{violet}, \eqref{paprika} one has
$$
\P( N_{(-\infty,x)}(\tilde M_n)=i \wedge N_{[x,x+s]}(\tilde M_n)=0 ) = 
\frac{1}{\sqrt{2\pi} \sigma} e^{-x^2/2\sigma^2} (\det(1 - 1_{[0,s]} P_\Dyson 1_{[0,s]}) + o(1)) + 
O( \log^{-0.85} n )$$
for $|x| \leq t_n$.  Since
$$ \int_{|x| \leq t_n} \frac{1}{\sqrt{2\pi} \sigma} e^{-x^2/2\sigma^2} = 1 - o(1)$$
we conclude from the choice of $t_n$ that
$$ \int_{|x| \leq t_n}\P( N_{(-\infty,x)}(\tilde M_n)=i \wedge N_{[x,x+s]}(\tilde M_n)=0 ) = 
\det(1 - 1_{[0,s]} P_\Dyson 1_{[0,s]}) + o(1)$$
and the claim \eqref{silly} then follows from \eqref{masx}.

It remains to establish the estimates \eqref{masx}, \eqref{indigo}, \eqref{violet}, \eqref{paprika}.  We begin with \eqref{masx}.  We can rewrite
$$ N_{(-\infty,x)}(\tilde M_n) = N_{(-\infty, \sqrt{n} u + \frac{x}{\sqrt{n} \rho_\sc(u)})}(M_n).$$
From the rigidity of eigenvalues of GUE (see\footnote{One can also derive this rigidity from the Bennett's inequality argument given below.  One could also use the rigidity results for more general Wigner matrices here, see \cite{EYY2} or \cite{TVcont}, though this would be overkill.} e.g. \cite[Corollary 5]{TVcont}) we know that 
$$ \P( N_{(-\infty,y)}(M_n) = i ) \ll n^{-100}$$
(say) unless
$$ y =\sqrt{n} u + O( \log^{O(1)} n / \sqrt{n} ) .$$
Because of this, to prove \eqref{masx} we may restrict to the regime where $x = O(\log^{O(1)} n)$.

By\footnote{Strictly speaking, Lemma \ref{dpp} is not applicable as stated because $(-\infty,y)$ is not a compact interval, but this can be addressed by the usual truncation argument, replacing $(-\infty,y)$ with $(-M,y)$ and then letting $M$ go to infinity, exploiting the exponential decay of $M_n$.  We omit the routine details.} Lemma \ref{dpp}, for any real number $y$, $N_{(-\infty,y)}(M_n)$ is the sum of $n$ independent Bernoulli variables.  The mean and variance of such random variables was computed in \cite{Gus}.  Indeed, from \cite[Lemma 2.1]{Gus} one has (after adjusting the normalisation)
$$ \E N_{(-\infty,y)}(M_n) = \int_{-\infty}^{y/\sqrt{n}} \rho_\sc(t)\ dt + O( \frac{\log n}{n} )$$
while from \cite[Lemma 2.3]{Gus} one has
$$ \Var N_{(-\infty,y)}(M_n) = (\frac{1}{2\pi^2} + o(1)) \log n.$$
Renormalising (and using the hypothesis $x = O(\log^{O(1)} n)$), we conclude that
$$ \E N_{(-\infty,x)}(\tilde M_n) = i + x + O(1)$$
and
$$ \Var N_{(-\infty,x)}(\tilde M_n) = (\frac{1}{2\pi^2} + o(1)) \log n.$$
Applying Bennet's inequality (see \cite{Ben62}), we conclude that
$$ \P( \E N_{(-\infty,x)}(\tilde M_n) = i ) \ll \exp( - c x / \sqrt{\log n} )$$
for some absolute constant $c>0$, which gives \eqref{masx}.  The bound \eqref{paprika} follows from the same computations, using Lemma \ref{dbe} (or Corollary \ref{core}) in place of Bennet's inequality.

The estimate \eqref{violet} is well known (see\footnote{Strictly speaking, Theorem 3.1.1 of \cite{AGZ} only treats the case $u=0$, but the general case $-2+\eps < u < 2-\eps$ follows from the same methods; see \cite[Exercise 3.7.5]{AGZ}.} e.g. \cite[Theorem 3.1.1]{AGZ}); for future reference we remark that this estimate also implies the crude lower bound
\begin{equation}\label{slog}
 \P( N_{[x,x+s](\tilde M_n)=0} ) \gg 1
\end{equation}
for $n$ sufficiently large.  We therefore turn to \eqref{indigo}.  By \eqref{slog} and \eqref{paprika},  it suffices to establish the conditional probability estimate
\begin{equation}\label{slime}
\P( N_{(-\infty,x)}(\tilde M_n)=i | N_{[x,x+s]}(\tilde M_n)=0 ) = 
\frac{1}{\sqrt{2\pi} \sigma} e^{-x^2/2\sigma^2} + O( \log^{-0.85} n ).
\end{equation}
We now turn to \eqref{slime}.  Recall that the eigenvalues of $M_n$ form a determinantal point process with kernel $K^{(n)}$ given by \eqref{kform}.  Rescaling this, we see that the eigenvalues of $\tilde M_n$ form a determinantal point process with kernel $\tilde K^{(n)}$ given by the formula
$$ \tilde K^{(n)}(x,y) := \frac{1}{\rho_\sc(u) \sqrt{n}} K^{(n)}( u \sqrt{n} + \frac{x}{\rho_\sc(u) \sqrt{n}}, u \sqrt{n} + \frac{y}{\rho_\sc(u) \sqrt{n}} ).$$
This is the kernel of an orthogonal projection $\tilde P^{(n)}$ to some $n$-dimensional subspace $\tilde V^{(n)}$ in $L^2(\R)$.  The elements of this subspace consist of polynomial multiples of a gaussian function, and in particular there is no non-trivial element of $\tilde V^{(n)}$ that vanishes on $[x,x+s]$.  Applying Corollary \ref{core} (and a truncation argument to deal with the non-compact nature of $(-\infty,x)$), one has
$$
\P( N_{(-\infty,x)}(\tilde M_n)=i | N_{[x,x+s]}(\tilde M_n)=0 ) = 
\frac{1}{\sqrt{2\pi} \sigma'} e^{-(i-\mu')^2/2(\sigma')^2} + O( (\sigma')^{-1.7} )$$
where
$$ \mu' := \tr(P' 1_{(-\infty,x)})$$
and
$$ (\sigma')^2 := \tr( P' 1_{(-\infty,x)^c} P' 1_{(-\infty,x)} ),$$
and $P'$ is the orthogonal projection to $1_{[x,x+s]^c} \tilde V^{(n)}$.  To establish \eqref{slime}, it will thus suffice to establish the bounds
$$ \mu' = O(1)$$
and
$$ (\sigma')^2 = \sigma^2 + O(1).$$

To do this, we will use Corollary \ref{closo}, with $J := [x,x+s]$, and the role of $P_0$ being played by the Dyson projection $P_\Dyson$.  From the well-known fact that a non-trivial function and its Fourier transform cannot both be compactly supported, we see that there is no non-trivial function in the range of $P_\Dyson$ supported in $J$.  As $P_\Dyson$ is locally trace class, we conclude from the Fredholm alternative (see e.g. \cite[Theorem VI.14]{rs}) that the compact operator $K_0$ defined by \eqref{pok} exists.  As $K_0$ is independent of $n$, we certainly have\footnote{Note that our bound here on $\|K_0\|_{op}$ is \emph{ineffective}, as it relies on the Fredholm alternative.  However, it is quite probable that one can obtain an effective bound on $K_0$ here by using a quantitative versions of Hardy's uncertainty principle to give a more robust version of the assertion that a non-trivial function and its Fourier transform cannot both be compactly supported. We will not pursue this issue here.}
$$ \|K_0\|_{op} \ll 1$$
and similarly
\begin{equation}\label{ja}
 \| 1_J P_\Dyson \|_{HS} \ll 1.
\end{equation}
By Corollary \ref{closo} (once again using a truncation argument to deal with the half-infinite nature of $(-\infty,x)$), it will thus suffice to show that
\begin{equation}\label{page}
 \| 1_J \tilde P^{(n)} \|_{HS} \ll 1
\end{equation}
and
\begin{equation}\label{pane}
\| (P_\Dyson - \tilde P^{(n)}) 1_J \|_{\op} = o(1).
\end{equation}
Since the Hilbert-Schmidt norm controls the operator norm, we see from \eqref{ja} that \eqref{page}, \eqref{pane} will both follow from the bound
$$
\| (P_\Dyson - \tilde P^{(n)}) 1_J \|_{HS} = o(1).$$
Using the integral kernels $K_\Dyson, \tilde K^{(n)}$ of $P_\Dyson$, $\tilde P^{(n)}$ and the compact nature of $J$, it suffices to show that
\begin{equation}\label{claim}
\int_\R |K_\Dyson(x,y) - \tilde K^{(n)}(x,y)|^2\ dx = o(1)
\end{equation}
uniformly for all $y \in J$.  In principle one could establish this bound from a sufficiently precise analysis of the asymptotics of Hermite polynomials (such as those given in \cite{dy}), but one can actually derive this bound from the standard convergence result \eqref{asym} as follows.  From \eqref{asym} we know that $\tilde K^{(n)}(x,y)$ converges locally uniformly in $x,y$ to $K_\Dyson(x,y)$ as $n \to\infty$, and so
\begin{equation}\label{tri}
 \int_{-L}^L |K_\Dyson(x,y) - \tilde K^{(n)}(x,y)|^2\ dx = o(1)
\end{equation}
for any fixed $L$.  Also, as $P_\Dyson, \tilde P^{(n)}$ are both projections, one has
$$ \int_\R |K_\Dyson(x,y)|^2\ dx = K_\Dyson(y,y)$$
and
$$ \int_\R |\tilde K^{(n)}(x,y)|^2\ dx = \tilde K^{(n)}(y,y).$$
From \eqref{asym}, one has
$$ \tilde K^{(n)}(y,y) = K_\Dyson(y,y) + o(1).$$
For any given $\eps > 0$, one can find an $L$ such that
\begin{equation}\label{lae}
 \int_{|x|>L} |K_\Dyson(x,y)|^2\ dx = O(\eps)
 \end{equation}
and thus
$$ \int_\R |\tilde K^{(n)}(x,y)|^2\ dx = \int_{-L}^L |K_\Dyson(x,y)|^2\ dx + O(\eps) + o(1).$$
But from \eqref{tri} and the triangle inequality we have
$$ \int_{-L}^L |\tilde K^{(n)}(x,y)|^2\ dx = \int_{-L}^L |K_\Dyson(x,y)|^2\ dx + o(1)$$
and so 
$$ \int_{|x|>L} |\tilde K^{(n)}(x,y)|^2\ dx = O(\eps) + o(1).$$
From this, \eqref{tri}, and \eqref{lae} we conclude that
$$ \int_\R |K_\Dyson(x,y) - \tilde K^{(n)}(x,y)|^2\ dx = O(\eps) + o(1)$$
and the claim \eqref{claim} follows by sending $\eps$ to zero.  The proof of Theorem \ref{gue-main} (and thus also Corollary \ref{wigner-main}) is now complete.

\end{document}